\documentclass[12pt, a4paper]{amsart}

\usepackage{geometry} 
\usepackage{fancyhdr}
\usepackage{url}
\usepackage{amssymb,amsthm}
\usepackage{amsmath}
\usepackage{accents}
\usepackage{eucal}
\usepackage{amscd}
\usepackage[shortalphabetic]{amsrefs}
\usepackage{mathptmx}
\usepackage{times}
\usepackage{graphics}
\usepackage{units}

\usepackage{enumerate}
\allowdisplaybreaks[1]

\newtheorem{theorem}{Theorem}[section]

\newtheorem{proposition}[theorem]{Proposition}

\newtheorem{thm}{Theorem}[section]
\newtheorem{remarks}[theorem]{Remarks}

\newtheorem{prop}[thm]{Proposition}

\newcommand{\circo}{\accentset{\circ}}

\providecommand{\abs}[1]{\lvert#1\rvert}

\DeclareMathOperator{\tr}{tr}

\newcommand{\s}{\gamma}

\newcommand{\p}{\partial}

\newcommand{\e}{\epsilon}
\newcommand{\al}{\alpha}
\newcommand{\nablap}{\accentset{\bot}{\nabla}}
\newcommand{\Rp}{Rm^{\perp}}

\newcommand{\Acirc}{\accentset{\circ}{A}}
\newcommand{\mc}{\mathcal}
\newcommand{\R}{\mathbb{R}}

\newcommand{\3}{1}
\newcommand{\4}{2}

\newcommand{\Kp}{K^{\perp}}

\begin{document}

\title{Surfaces of Co-Dimension Two Pinched by Normal Curvature}



\author{Charles Baker}
\author{Huy The Nguyen}
\address{
The University of Queensland,
St Lucia, Qld 4072, Australia}
\email{Charles.Baker@uqconnect.edu.au}
\email{huy.nguyen@maths.uq.edu.au}


\begin{abstract}
We prove that codimension two surfaces satisfying a nonlinear curvature condition depending on normal curvature are smoothly deformed by mean curvature flow to round points.
\end{abstract}

\maketitle
\section{Introduction}

We consider two dimensional surfaces of codimension two immersed in Euclidean four-space, which includes, for example, the Clifford torus when viewed as a submanifold of $\mathbb{R}^4$. The main theorem we present asserts that surfaces satisfying a curvature pinching depending on normal curvature are deformed by the mean curvature flow to round points. In contrast to hypersurfaces, very little progress has been made on mean curvature flow in high codimension owing to the nontrivial structure of the normal bundle. The best result to date is due to Andrews and Baker \cite{Andrews2010}, where it is shown that, for suitable values of a constant $k$ depending on dimension but not codimension, submanifolds satisfying the pinching condition $|A|^2 \leq k |H|^2$ evolve under the mean curvature flow to round points, which can be considered a high codimension analogue of Huisken's seminal result on mean curvature flow of hypersurfaces \cite{Huisken1984}. In this paper we show for the first time that inclusion of normal curvature in the pinching cone provides improved geometric estimates, expanding the class of surfaces known to be diffeomorphic to round spheres.

The submanifold estimates are much more difficult than their hypersurface counterparts, being complicated by the presence of normal curvature. The main theorem of \cite{Andrews2010} is optimal for submanifolds of dimension four and greater (independent of the codimension), where the tori $\mathcal{S}^{n-1}(\epsilon) \times \mathcal{S}(1) \subset \mathbb{R}^{n} \times \mathbb{R}^{2}$ are obstructions to improving the pinching constant beyond $1/(n-1)$. The theorem is suboptimal in dimensions two and three, with pinching constant $k=4/(3n)$, because of unfavourable reaction terms. With the inclusion of normal curvature, the new pinching condition turns out to be optimal for the reaction terms, but the gradient terms still obstruct the attainment of optimal pinching, similar to the flow of hypersurfaces in a spherical background \cite{Huisken1987}. The main result we obtain in this article is as follows:
\begin{thm}\label{thm:mainThm}
Suppose $\Sigma_0=F_0(\Sigma^2)$ is a closed surface smoothly immersed in $\mathbb{R}^{4}$.  If $\Sigma_0$ satisfies $\abs{ H }_{ \text{min} } > 0$ and $\abs{A}^2 +2\gamma|K^{\bot}| \leq k\abs{H}^2$, where $ \gamma = 1 - \nicefrac43 k$ and $k \leq \nicefrac {29}{40}$,
then the mean curvature flow of $\Sigma_0$ has a unique smooth solution $\Sigma_t$ on a finite maximal time interval $t \in [0,T)$. There exists a sequence of rescaled mean curvature flows $F_j : \Sigma^2 \times I_j \rightarrow \mathbb{R}^{4}$ containing a subsequence of mean curvature flows (also indexed by $j$) that converges to a limit mean curvature flow $F_{\infty} : \Sigma^2_{\infty} \times (-\infty, 0] \rightarrow \mathbb{R}^{4}$ on compact sets of $\mathbb{R}^{4} \times \mathbb{R}$ as $j \rightarrow \infty$. Moreover, the limit mean curvature flow is a shrinking sphere.
\end{thm}
This theorem improves the pinching constant of \cite{Andrews2010} from $\nicefrac23$ to $\nicefrac34 - \nicefrac1{40}$, which is, similar to the hypersurface theory, almost the best constant thought to be achievable with the mean curvature flow. The inclusion of normal curvature in the pinching condition cancels the unfavourable reaction terms encountered in \cite{Andrews2010}, however, the gradient of the normal curvature prohibits pushing the pinching constant all the way to $\nicefrac34$. We conjecture that the Clifford torus viewed as a two-surface of codimension two in $\mathbb{R}^4$ is the true obstruction to the theorem, corresponding to an optimal pinching constant of $k=1$. The Clifford torus is still intrinsically flat but no longer minimal in $\mathbb{R}^4$ and satisfies $|A|^2 = |H|^2$.

We take this opportunity to announce another new result of independent interest, discovered in the course of estimating the nonlinearity in the Simons identity (see Proposition \ref{prop:newSimonsId}). Obtaining a positive lower bound on the nonlinearity in Simons' identity is a crucial step in the integral estimates used to prove convergence to a round point. In the case of two-surfaces of codimension two (in this case immersed in a Euclidean background), it is possible to compute the nonlinearity exactly, with the result that $Z =  2K|\Acirc|^2 - 2|K^{\perp}|^2$. The Simons identity plays a key role in a series of classification results, initiated in a famous paper by Chern, do Carmo and Kobayashi, \cite{ChernCarmoKobayashi1970}, where it is proved that if a $n$-dimensional submanifold of a $(n+p)$-dimensional sphere satisfies $|A|^2 \leq n/(1-1/p)$, then the submanifold is totally geodesic, or if the equality holds identically, then it is the Clifford torus or Veronese surface. With our refined understanding of the Simons identity nonlinearity we are able to provide a new classification result depending not on the length of the second fundamental form, but rather on a pointwise pinching of the intrinsic and normal curvatures.

\begin{theorem}
Suppose a two surface $\Sigma^2$ minimally immersed in $\mathcal{S}^4$ satisfies $|K^\perp|^2 \leq K | \Acirc|^2$. Then either 
\begin{enumerate} [i)]
\item $|A| ^ 2 \equiv 0 $ and the surface is a geodesic sphere; or

\item $|A| ^ 2 \not \equiv 0 $, in which case either
\begin{enumerate}
\item $|K^{\perp}| = 0$, and the surface is the Clifford torus, or

\item $K^{\perp} \neq 0$ and it is the Veronese surface.
\end{enumerate}
\end{enumerate}

\end{theorem}
In a sequel to this paper we investigate submanifolds immersed in a spherical background moving by the mean curvature flow, where a proof of the above theorem more naturally resides. The argument involves careful examination of the curvature terms and an application of the strong maximum principle.

\section{Notation and preliminary results}
We adhere to the notation of \cite{Andrews2010} and in particular use the canonical space-time connections introduced in that paper. A fundamental ingredient in the derivation of the evolution equations is Simons' identity:
\begin{equation}\label{eqn:SimonsId}
\Delta h_{ij}=\nabla_i\nabla_jH+H\cdot h_{ip}h_{pj}-h_{ij}\cdot h_{pq}h_{pq}+2h_{jq}\cdot h_{ip}h_{pq}
-h_{iq}\cdot h_{qp}h_{pj}-h_{jq}\cdot h_{qp}h_{pi}.
\end{equation}
The timelike Codazzi equation combined with Simons' identity produces the evolution equation for the second fundamental form:
\begin{equation}\label{eqn:evolA}
\nabla_{\partial_t}h_{ij} = \Delta h_{ij}+h_{ij}\cdot h_{pq}h_{pq}+h_{iq}\cdot h_{qp}h_{pj}
+h_{jq}\cdot h_{qp}h_{pi}-2h_{ip}\cdot h_{jq} h_{pq}.
\end{equation}
The evolution equation for the mean curvature vector is found by taking the trace with $g_{ij}$:
\begin{equation}
\nabla_{\partial_t}H=\Delta H + H\cdot h_{pq}h_{pq}.
\end{equation}
The evolution equations of the squared lengths of the second fundamental form and the mean curvature vector are
\begin{align}
	\frac{ \p }{ \p t} \abs{A}^2 &= \Delta\abs{A}^2 - 2\abs{\nabla A}^2 + 2 \sum_{\alpha, \beta} \bigg( \sum_{i,j} h_{ij\alpha}h_{ij\beta} \bigg)^2 +2 \sum_{i,j,\alpha,\beta} \bigg( \sum_p h_{ip\alpha}h_{jp\beta} - h_{jp\alpha}h_{ip\beta} \bigg)^2 \label{eqn:A2}\\
	\frac{\partial}{\partial t}\abs{H}^2 &= \Delta\abs{H}^2 - 2\abs{\nablap H}^2 + 2\sum_{i,j} \bigg( \sum_{\alpha} H_{\alpha}h_{ij\alpha}\bigg)^2. \label{eqn:H2}
\end{align}
The last term in \eqref{eqn:A2} is the squared length of the normal curvature, which we denote by $\abs{ \Rp }^2$.  For convenience, we label the reaction terms of the above evolution equations by
\begin{gather*}
	R_1 = \sum_{\alpha, \beta} \bigg(\!\sum_{i,j} h_{ij\alpha}h_{ij\beta}\!\bigg)^2 + \abs{\Rp}^2 \\
	R_2 = \sum_{i,j}\!\bigg(\!\sum_{\alpha} H_{\alpha}h_{ij\alpha}\bigg)^2.
\end{gather*}

The following existence theorem holds for the mean curvature flow of $\Sigma_0$ under the conditions of Theorem \ref{thm:mainThm}:
\begin{thm}\label{thm:longTimeExistence}
The mean curvature flow of $\Sigma_0$ exists on a finite maximal time interval $0 \leq t < T < \infty$.  Moreover, $\max_{\Sigma_t} \abs{A}^2 \rightarrow \infty$ as $t \rightarrow T$.
\end{thm}
The proof that the maximal time of existence is finite follows easily from the evolution equation for the position vector $F$: $\frac{\partial}{\partial t}\abs{F}^2 = \Delta\abs{F}^2-2n$.  The maximum principle implies $\abs{F(p,t)}^2\leq R^2-2nt$ and thus $T\leq \nicefrac{R^2}{2n}$, where $R=\max\left\{\abs{F_0(p)}:\ p\in\Sigma\right\}$. The proof of the second part of the theorem can be found in \cite{Andrews2010}.

\section{Evolution of normal curvature}
In this section we compute the evolution equation for the normal curvature. The normal curvature tensor in local orthonormal frames for the tangent $\{ e _i  : i = 1,2 \}$ and normal $\{ \nu _ \alpha : \alpha = 1,2 \}$ bundles is given by 
\begin{align} \label{evol_normal}
R ^ {\perp} _{ ij\alpha\beta} = h _{ i p \alpha } h _{ jp \beta } - h _ {jp \alpha } h _{ ip \beta}.
\end{align}
We will often compute in a local orthonormal normal frame $\{ \nu _ \alpha : \alpha = 1,2 \}$ where $\nu_1 = \nicefrac { H} { | H| }$. As the normal bundle is two dimensional $ \nu_2 $ is then determined by $ \nu _ 1 $ up to sign. With this choice of frame the second fundamental form becomes
\begin{align}\label{eqn_traceless} 
\left \{ 
\begin{array}{cc} 
\circo A _ 1 = A _ 1 - \frac {|H|} {n} Id \\
\circo A_ 2 = A_2
\end{array}
\right.
\end{align}
and 
\begin{align*}
\left \{ 
\begin{array}{cc}
\tr A_1 = | H| \\
\tr A_2 = 0.
\end{array}
\right.
\end{align*}
It is also always possible to choose the tangent frame $\{ e _i  : i = 1,2 \}$ to diagonalise $A_1$. We often refer to the orthonormal frame $ \{ e_1, e_2, e_3, e_4\} = \{ e_1, e_2, \nu_1, \nu _2\}$, where $\{e_i\}$ diagonalises $A_1$ and $\nu_1 = H/|H|$, as the `special orthonormal frame'.
Codimension two surfaces have four independent components of the second fundamental form, which makes it still tractable to work with individual components, similar to the role of principal curvatures in hypersurface theory. Working in the special orthonormal frame, we often find it convenient to represent the second fundamental form by
\begin{align}\label{eqn_split}
h_{ ij } = \left [
\begin{array}{cc}
 \frac { |H|} {2} + a  & 0 \\
 0 & \frac { |H| } {2} - a 
\end{array}
  \right]\nu_1
  + \left[ 
  \begin{array}{cc} 
  b & c \\
  c & -  b
  \end{array}
  \right]\nu_2,
\end{align}
so that $h_{111} = |H|/2 +a$, $h_{221} = |H|/2-a$, $h_{112} = b$, $h_{122} = c$ and so on. Note that $|\Acirc|^2 = 2a^2 + 2b^2 +2c^2$.

Just as a surface has only one sectional curvature $K$, a codimension two surface also has only one normal curvature, which we denote by $ K ^ \perp$. In the special orthonormal frame the normal curvature is
\begin{equation}\label{eqn:normalCurv}
\begin{split}
K^\perp = R^{\perp}_{1234} & = \sum_{p}\left(h_{1p\3 }h_{2p\4 } - h_{2p\3 }h_{1p\4 }\right)\\
& = h_{11\3 }h_{21\4 } - h_{21\3 }h_{11\4 } +h_{12\3 }h_{22\4 } - h_{22\3 }h_{12\4 }\\
& = (h_{11\3 } - h_{22\3 })h_{12\4 }+ h_{12\3 }( h _{ 22\4  } - h_{ 11\4  }) \\
& = 2 a c.
\end{split}
\end{equation}
Note also that $| \Rp | ^ 2 = 16 a ^ 2 c ^ 2 $. Differentiating \eqref{evol_normal} and using equation \eqref{eqn:evolA} we have
\begin{align*}
\frac{\partial}{\partial t} R ^ { \perp} _{ ij\alpha \beta } &= \Delta R ^ { \perp} _{ ij \alpha \beta } - 2 \sum _{p,r }\left ( \nabla_q h _{ ip\alpha} \nabla _ q h _{ jp\beta} - \nabla _ q h _{ jp \alpha } \nabla _ q h _{ i p \beta }\right ) \\
 &\quad +\sum_{ p } \left (\frac{ d}{ dt} h _{ i p \alpha} h _{ j p \beta } + h _{ i p \alpha} \frac{ d } {dt } h _{ j p \beta } - \frac { d} { dt } h _{ j p \alpha } h _ {i p \beta } - h _{ jp \alpha} \frac{ d}{ dt }h_{ i p \beta }\right)
\end{align*}
or 
\begin{equation}\label{eqn:evolNormalCurv}
\begin{split}
\frac{\partial}{\partial t}   R ^ { \perp} _{ ij\alpha \beta } & = \Delta R ^ { \perp} _{ ij \alpha \beta } - 2 \sum _{p,r }\left ( \nabla_q h _{ ip\alpha} \nabla _ q h _{ jp\beta} - \nabla _ q h _{ jp \alpha } \nabla _ q h _{ i p \beta }\right ) \\
&\quad + \sum (h _{ ip \gamma } \cdot h_{rq\gamma} h _{ rq\al } + h _{ i q \gamma } \cdot h _{ q r \gamma} h _{ rp \al}+ h _{ p q\gamma} \cdot h_{ qr\gamma } h _{ ri \al} - 2 h _{ ir \gamma} \cdot h _{ pq\gamma } h _{ rq\al })h_{jp\beta}\\
&\quad+ \sum h_{ip\alpha}(h _{ jp \gamma } \cdot h_{rq\gamma} h _{ rq\beta } + h _{ j q \gamma } \cdot h _{ q r \gamma} h _{ rp \beta}+ h _{ p q\gamma} \cdot h_{ qr\gamma } h _{ rj \beta} - 2 h _{ jr \gamma} \cdot h _{ pq\gamma } h _{ rq\beta }) \\
&\quad-\sum ( h _{ jp \gamma } \cdot h_{rq\gamma} h _{ rq\al } + h _{ j q \gamma } \cdot h _{ q r \gamma} h _{ rp \al}+ h _{ p q\gamma} \cdot h_{ qr\gamma } h _{ rj \al} - 2 h _{ jr \gamma} \cdot h _{ pq\gamma } h _{ rq\al }) h_{ ip \beta}\\
&\quad - \sum h_{jp\al}(h _{ ip \gamma } \cdot h_{rq\gamma} h _{ rq\beta } + h _{ i q \gamma } \cdot h _{ q r \gamma} h _{ rp \beta}+ h _{ p q\gamma} \cdot h_{ qr\gamma } h _{ ri \beta} - 2 h _{ ir \gamma} \cdot h _{ pq\gamma } h _{ rq\beta } ).
\end{split}
\end{equation}
Using the special orthonormal frame the nonlinearity for codimension two surfaces simplifies to
\begin{align*}
 \frac{ d}{ dt} K ^ \perp &= 4 ac \left(  \left ( \frac{ | H|}{2} - a \right)^ 2 - \left (\frac { |H|} { 2 } + a  \right) \left ( \frac { | H|} { 2 } - a \right) + 2 b ^ 2 + 3 c ^ 2 +\left ( \frac { |H|}{ 2 }  + a \right) ^ 2\right) \\
  &= K^{\perp}\left( |A|^2 + 2|\Acirc|^2 - 2b^2 \right). 
\end{align*}
For notational convenience we set
\[ \nabla_{\!\!evol}K^{\perp} := \sum _{p,r }\left ( \nabla_q h _{ ip\alpha} \nabla _ q h _{ jp\beta} - \nabla _ q h _{ jp \alpha } \nabla _ q h _{ i p \beta }\right ) \]
and
\[ R_3 := K^{\perp}\left( |A|^2 + 2|\Acirc|^2 - 2b^2 \right). \]
Substituting the simplifed nonlinearity into \eqref{eqn:evolNormalCurv} we obtain the evolution equation for the normal curvature
\[ \frac{\p  }{ \p t } K^{\perp}= \Delta K^{\perp} - 2 \nabla_{\!\!evol}K^{\perp} + K^{\perp}\left( |A|^2 + 2|\Acirc|^2 - 2b^2 \right), \]
and a little more computation shows the length of the normal curvature evolves by
\[ \frac{\p  }{ \p t } |K^{\perp}|= \Delta |K^{\perp}| - 2 \frac{ K^{\perp} }{|K^{\perp}|  }\nabla_{\!\!evol}K^{\perp} + |K^{\perp}|\left( |A|^2 + 2|\Acirc|^2 - 2b^2 \right). \]
We remark that the complicated structure of the gradient terms prevents an application of the maximum principle to conclude flat normal normal bundle is preserved.

\section{Preservation of curvature pinching}\label{sec:preservation}
The first step towards Theorem \ref{thm:mainThm} is to show a certain quadratic curvature condition involving the normal curvature is preserved by the mean curvature flow. Note that due to the $\epsilon$ in the following proposition, as an automatic corollary we see that $|H| > 0$ is also preserved along the flow.
\begin{prop}\label{prop:pinching}
If a solution $ F: \Sigma \times [0,T) \rightarrow \R^{4}$ of MCF satisfies $ |A| ^ 2 + 2 \gamma | K^\perp| + \epsilon < k |H|^ 2 $
where $ \gamma = 1 -\nicefrac43 k $ and $ \nicefrac12 <k < \nicefrac{ 29 }{ 40 } $ then this remains true for all $ 0\leq t <  T$.
\end{prop}
With exception of the last estimate, the following gradient estimates are well-known; the third estimate is new.
\begin{proposition}\label{prop_grad}
We have the following gradient estimates:
\begin{subequations}
\begin{align}
|\nabla A | ^ 2 &\geq \frac  {3 }{ n+2 } | \nabla H | ^ 2  \label{eqn_gradient1} \\
|\nabla A | ^ 2 - \frac  {1 }  {n} |\nabla H| ^ 2 &\geq \frac  {2 ( n-1)}{3 n} | \nabla A | ^ 2 \label{eqn_gradient2}\\
|\nabla A| ^ 2 & \geq  2 \nabla_{\!\!evol} K^{\perp} \quad \text{if $ n =2 $ }. \label{eqn_gradient3}
\end{align}
\end{subequations}
\end{proposition}
\begin{proof} 
The first two inequalities are proven in \cite{Huisken1984}, motivated by similar estimates in the Ricci flow \cite{Hamilton1982}. They are established by decomposing the tensor $ \nabla A $ into orthogonal components $\nabla _i  h _{ jk } = E _{ ijk} + F _{ ijk}$,
where 
\begin{align*}
 E_{ ijk} = \frac  {1 }{ n+2 } ( g _{ij} \nabla _k H + g_{ ik} \nabla _j H + g_{ jk} \nabla _i H ),
\end{align*}
from which it follows that $ |\nabla A | ^ 2 \geq | E | ^ 2 = \frac { 3 }{ n+ 2 } | \nabla H | ^ 2 $. The second estimate follows from the first. In order to prove the third inequality, we use the Codazzi equation to evaluate
\begin{multline*}
\sum _{p,q }\left ( \nabla_q h _{ 1p1} \nabla _ q h _{ 2p2} - \nabla _ q h _{ 2p 1 } \nabla _ q h _{ 1p 2 }\right )
= \nabla _{ 1 } h _{ 111} \nabla_1 h_{ 122} - \nabla _ 1 h_{ 112} \nabla _ 2 h _{ 111} + 2 \nabla _1 h _{ 222} \nabla _ 2 h _{ 111}  \\- 2 \nabla _ 1 h _{ 122} \nabla _ 1 h _{ 222} + \nabla _ 1 h _{ 221} \nabla _2 h _{ 222} - \nabla _ 1 h _{ 222} \nabla _ 2 h _{ 221}.
\end{multline*}
Writing down all the terms in $ |\nabla A | ^ 2 $ we get
\begin{align*}
|\nabla A| ^ 2 & =( \nabla _ 1 h _{ 111} )^ 2 + 3( \nabla _ 2 h _{ 111} )^ 2 + 3( \nabla _ 1 h _{ 122} )^ 2 \\
&+ 3( \nabla _{ 1 } h _{ 222} )^ 2 +( \nabla _2 h_{ 221} )^ 2 + (\nabla _ 2 h _{ 222} )^ 2 + 3 (\nabla _ 1 h _{ 221})^ 2 +(\nabla_1 h_{112}) ^2,
\end{align*}
and the estimate follows by applying the Cauchy-Schwarz inequality and comparing terms.
\end{proof}

\begin{proof}[Proof of Proposition 4.1]
Suppose the submanifolds satisfies $| A| ^ 2 +2 \s | K ^ \perp|  - k |H| ^ 2 < 0$ at the initial time. As the submanifold is compact and the inequality is strict, we can find an $\epsilon > 0$ such that $\mc Q := | A| ^ 2 +2 \s | K ^ \perp|  - k |H| ^ 2 + \epsilon< 0$ also holds at the initial time. Combining the evolution equations for $|A|^2$, $|K^{\perp}|$ and $|H| ^2 $ we have
\begin{equation*}
\frac { \partial }{ \partial t } \mc{Q} = \Delta \mc{Q} - 2 \left( |\nabla A | ^ 2 + 2\gamma \frac{\Kp}{|\Kp|}\nabla_{evol} \Kp - k|\nabla H| ^ 2 \right) \\ + 2 R_1 + 2\gamma R_3 - 2 kR  _2.
\end{equation*}
We deal with the gradient terms first. Using the the gradient estimates \eqref{eqn_gradient1} and \eqref{eqn_gradient3} we have
\begin{align*}
-2 \left( |\nabla A | ^ 2 + 2\gamma\frac{\Kp}{|\Kp|}\nabla_{evol} \Kp - k |\nabla H| ^ 2 \right) &\leq \left(- 2 + 2\gamma + 2\frac43k \right) |\nabla A|^2,
\end{align*}
which is less than zero provided $\gamma < (1-4/3k)$.

Next we deal with the reaction terms:
\begin{align}
\frac{d}{dt} \mc Q &= 2\sum _{ \alpha ,\beta}\bigg(\sum_{ i, j } h _{ ij\alpha} h_{ ij \beta } \bigg)^ 2  + 2| \Rp  | ^ 2 - 2 k \sum_{ i, j} \left ( \sum_{\alpha } H _ \alpha h_{ ij \alpha }\right ) ^ 2 + 2\gamma R_3   \nonumber \\
&= 2 | \Acirc_1 | ^ 4  - 2 \left( k - \frac { 2 } { n} \right) | \Acirc_1 | ^ 2 | H| ^ 2 - \frac 2 n \left( k - \frac { 1 }{ n} \right) | H| ^ 4 \nonumber \\
&\quad + 4 \bigg( \sum_{ i,j} \circo h _{ ij1}\circo  h _{ ij2 }\bigg)^2  + 2 \bigg( \sum_{ i ,j } \circo h_{ ij 2 } \circo h_{ ij 2 } \bigg) ^ 2 + 2 | \Rp | ^ 2 \nonumber \\
&\quad + 2\gamma|K^{\perp}|\left( |A|^2 + 2|\Acirc|^2 - 2b^2 \right). \label{eqn:Qrxn}
\end{align}
Written in the terms of the special orthonormal frames, the bracketed terms on the second last line above are
\begin{align*}
4  \bigg( \sum_{ i,j} \circo h _{ ij1}\circo  h _{ ij2 }\bigg)^2 & = 16 a ^ 2 b^ 2, \quad 2 \bigg( \sum_{ i ,j } \circo h_{ ij 2 } \circo h_{ ij 2 } \bigg) ^ 2 =  2 ( 2 b ^ 2 + 2 c ^ 2 ) ^ 2. 
\end{align*}
Now suppose, for a contradiction, that there exists a first point in time where $\mc Q = 0$. Computing at this point, as $\mc Q = 0$ we have $ \left( k - \frac 1 n \right) | H |^ 2 = (|\Acirc |^ 2 + 2 \s  | K^ \perp | +\e)$, and substituting this into \eqref{eqn:Qrxn} to eliminate the $|H|^2$ terms we obtain after some computation
\begin{equation}\label{eqn:Qrxn}
\begin{split}
\frac{d}{dt} \mc Q &= \left(-\frac{1}{k-1/2} + 2 \right)4 a^2b^2 + \left(-\frac{1}{k-1/2} + 2 \right) \gamma|\Kp||\Acirc_1|^2  \\
&\quad +  \left(-\frac{3}{k-1/2} + 6 \right) \gamma|\Kp||\Acirc_2|^2 + \left(-\frac{1}{k-1/2} + 2 \right)|\Acirc_2|^4 \\
&\quad + \left(-\frac{(1 + 2\gamma^2)}{k-1/2} + 6 \right)|\Kp|^2 \\
&\quad - \epsilon\left( 2 + \frac{1}{k-1/2} \right) |\Acirc_1|^2 - \frac{2\epsilon}{k-1/2}|\Acirc_2|^2 - \frac{3\epsilon\gamma|\Kp|}{k-1/2} - \frac{\epsilon^2}{k-1/2}.
\end{split}
\end{equation}
where we have we used $|\Acirc_1|^2 |\Acirc_2|^2 = 4a^2b^2 + |\Kp|^2$. With the exception of the $|\Kp|^2$ term, all terms are negative provided $k < 1$, which is the best constant we expect. We group the remaining terms into two quadratic forms to exploit the negative terms to control the $|\Kp|^2$ term, the most restrictive term. Discarding the negative terms not useful in controlling normal curvature, expanding and grouping terms we have
\begin{align*}
\frac{d}{dt} \mc Q &\leq 4c^2 \left\{ \left(-\frac{1}{k-1/2} + 2 \right)c^2 + \eta_1\left(-\frac{3}{k-1/2} + 6 \right) \gamma |ac|
 + \eta_2 \left(-\frac{(1 + 2\gamma^2)}{k-1/2} + 6 \right) a^2 \right\} \\
&\quad + 4|ac| \bigg\{ \left(-\frac{1}{k-1/2} + 2 \right)\gamma a^2 + (1-\eta_2)\left(-\frac{(1 + 2\gamma^2)}{k-1/2} + 6 \right) |ac| \\
&\quad + (1-\eta_1)\left(-\frac{3}{k-1/2} + 6 \right) \gamma c^2 \bigg\}.
\end{align*}
We now substitute $\gamma = 1 - 4/3 k - \delta$ in order to keep the gradient term negative, and use the parameters $\eta_1, \eta_2$ to shift as much bad normal curvature into the first curly bracket to consume all of the good $c^4$ term. As it does not seem possible to reach $k = 3/4$, we have been satisfied to numerical explore the parameter values, with the result that the above term is strictly negative for $k = 29/40$. Choosing $ k \leq \nicefrac { 29}{40 }$ and $ \gamma = 1 - \frac 4 3 k - \delta$ ensures that both the gradient and reaction terms of the evolution equation for $\mc Q$ are negative, which via the maximum principle provides a contradiction, and we conclude $\mc Q < 0$ is preserved by the flow.

\end{proof}
\begin{remarks}
In fact, not taking into account the gradient term so that we can choose $\gamma$ and $k$ independently, referring to \eqref{eqn:Qrxn}, we see the nonlinearity is non-positive if $ \gamma =1 ,k =1 $, which equates to $ | A | ^ 2 + 2 | K ^ \perp | \leq | H| ^ 2 $ or equivalently $|\Kp| \leq K$. This estimate for the nonlinearity is optimal, as the Clifford torus embedded in $ \R ^4 $ satifies as $ K ^ \perp =0 $ and $ |A| ^ 2 = |H| ^ 2 $. 
\end{remarks}
\section{Improvement of curvature pinching}
In the previous section we saw that curvature pinching is preserved by the mean curvature flow. In this section, we prove that curvature pinching actually improves along the flow. We show that in regions where mean curvature becomes large, the evolving surface becomes increasingly totally umbilic, ultimately allowing us to conclude convergence to a sphere.

\begin{theorem}\label{eqn:tracelessEst}
There exists constants $c_0 < \infty $ and $ \delta > 0$ both depending only on $ \Sigma_0$ such that for all time $ t \in [0, T)$ we have the estimate 
\begin{equation}
|\Acirc|^ 2 + 2\gamma |K^{\perp}| \leq  c_ 0 |H| ^{ 2 -\delta}. 
\end{equation}
\end{theorem}

We seek to bound the function $f_{\sigma} := ( |\Acirc|^2 + 2\gamma|K^{\perp}| )/ |H|^{2(1-\sigma)}$, where $\sigma > 0$ and small. The reaction terms of the evolution equation for $f_{\sigma}$ contain a small positive quantity, impeding the use of the maximum principle to conclude the desired result. Following Huisken \cite{Huisken1984}, we proceed by exploiting a favourable gradient term with a Poincar\'e-type inequality and bounding $f_{\sigma}$ in $L^{\infty}$ by a Stampacchia iteration procedure. The derivation of the Poincar\'e-type inequality from integrating Simons' identity and the Stampacchia iteration are well-known in the mean curvature flow literature, however, in our case, we must also control the normal curvature.

\begin{proposition} 
For every $\sigma \in (0,1)$ and $\epsilon_{\nabla} :=  1 - \nicefrac43k - \gamma$ we have the evolution equation
\begin{align*}
\frac \partial { \partial t } f _{ \sigma} \leq \Delta f _\sigma + \frac{ 2 ( 1- \sigma) }{ |H|^2 } \langle \nabla _i |H|^2, \nabla _ i f _ \sigma \rangle - \frac { 2 \epsilon _\nabla } { |H| ^ { 2 ( 1-\sigma)} }| \nabla A | ^ 2 + 2 \sigma | A| ^ 2 f _ \sigma.
\end{align*}
\end{proposition}

\begin{proof}
Differentiating $f_{\sigma}$ in time and substituting in the relevant evolution equations we get
\begin{equation}\label{e: evol eqn f_sigma 2}
	\begin{split}
		\frac{\partial }{\partial t}f_{\sigma} &= \frac{\Delta\abs{A}^2 - 2\abs{\nabla A}^2 + 2R_1 }{ (\abs{H}^2)^{1-\sigma} }  +   \frac{ 2\gamma \left( \Delta\abs{ K^{\perp} } + 2 \Kp/|\Kp| \nabla_{evol} K^{\perp}  + R_3 \right)}{ (\abs{H}^2)^{1-\sigma} }  \\ 
		&\quad- \frac{1}{n}\frac{ ( \Delta\abs{H}^2 - 2\abs{\nabla H}^2 + 2R_2 ) }{ ( \abs{H}^2)^{1-\sigma} } \\
		&\quad- \frac{ (1-\sigma)(\abs{A}^2 + 2\gamma|K^\perp| - \nicefrac{1}{n}\abs{H}^2) }{ (\abs{H}^2)^{2-\sigma} }( \Delta\abs{H}^2 - 2\abs{\nabla H}^2 + 2R_2).
	\end{split}
\end{equation}
After some computation, we find the Laplacian of $f_{\sigma}$ is
\begin{equation}
\begin{split}
\Delta f_{\sigma} &= \frac{ \Delta ( \abs{A}^2 + 2\gamma|K^{\perp}| - \nicefrac{1}{n}\abs{H}^2 ) }{ ( \abs{H}^2 )^{1-\sigma} } - \frac{ 2(1-\sigma) }{ ( \abs{H}^2 )^{2-\sigma} } \big\langle\nabla_i (\abs{A}^2 + 2\gamma|K^{\perp}| - \nicefrac{1}{n}\abs{H}^2), \nabla_i\abs{H}^2 \big\rangle  \\
		&\quad - \frac{ (1-\sigma)(\abs{A}^2 + 2\gamma|K^{\perp}| - \nicefrac{1}{n}\abs{H}^2) }{ ( \abs{H}^2 )^{2-\sigma} } \Delta\abs{H}^2 \\
		&\quad+ \frac{ (2-\sigma)(1-\sigma)(\abs{A}^2 + 2\gamma|K^\perp| - \nicefrac{1}{n}\abs{H}^2 }{ ( \abs{H}^2 )^{3-\sigma} } \abs{\nabla\abs{H}^2}^2,
\end{split}
\end{equation}
and the gradients satisfy
\begin{multline}
-\frac{ 2(1-\sigma) }{ ( \abs{H}^2 )^{2-\sigma} } \big\langle \nabla_i (\abs{A}^2 +2\gamma|K^\perp| - \nicefrac{1}{n}\abs{H}^2), \nabla_i\abs{H}^2 \big\rangle \\ = -\frac{ 2(1-\sigma) }{ \abs{H}^2 } \big\langle \nabla_i\abs{H}^2, \nabla_i f_{\sigma} \big\rangle - \frac{ 2(1-\sigma)^2 }{ (\abs{H}^2)^2 } f_{\sigma}\abs{ \nabla\abs{H}^2 }^2.
\end{multline}
With the aid of the last two formulae above, equation \eqref{e: evol eqn f_sigma 2} can be manipulated into the form
\begin{equation*}
	\begin{split}
		\frac{\partial}{\partial t} f_{\sigma} &= \Delta f_{\sigma} + \frac{ 2(1-\sigma) }{ \abs{H}^2 } \big\langle \nabla_i\abs{H}^2, \nabla_i f_{\sigma} \big\rangle + \frac{ 2\sigma R_2 f_{\sigma} }{ \abs{ H }^2 } \\
		&\quad - \frac{ 2 }{ (\abs{H}^2)^{1-\sigma} } \left( \abs{\nabla A}^2 + 2\gamma \Kp/|\Kp|\nabla_{\!\!evol} K^{\perp} - \frac{ \abs{A}^2 +2\gamma\abs{ K^{\perp} } }{ \abs{H}^2 }\abs{ \nabla H }^2 \right)  \\
		&\quad - \frac{ \sigma(1-\sigma) }{ (|H|^2)^2 } f_{\sigma}\abs{ \nabla\abs{H}^2 }^2 - \frac{ 2\sigma(\abs{A}^2  + 2\gamma\abs{ K^{\perp} } - \nicefrac{1}{n}\abs{H}^2) }{ (\abs{H}^2)^{2-\sigma} }\abs{\nabla H}^2 \\
		&\quad+ \frac{2}{ (\abs{H}^2)^{1-\sigma} }\left( R_1 + \gamma R_3 - \frac{ \abs{A}^2 + 2\gamma\abs{ K^{\perp} } }{ \abs{H}^2 }R_2 \right).
	\end{split}
\end{equation*}
We discard the terms on the last two lines as these are non-positive under our pinching assumption. We estimate the last term on the first line by $R_2 \leq \abs{A}^2\abs{H}^2$, and the gradient terms on the second line by
\begin{multline}
	-\frac{ 2 }{ (\abs{H}^2)^{1-\sigma} } \left( \abs{\nabla A}^2 + 2\gamma\Kp/|\Kp|\nabla_{\!\!evol} K^{\perp} - \frac{ \abs{A}^2 +2\gamma\abs{ K^{\perp} } }{ \abs{H}^2 }\abs{ \nabla H }^2 \right)  \\ \leq -\frac{ 2 ( 1 - \nicefrac43k - \gamma)  }{ (\abs{H}^2)^{1-\sigma} } 
	\abs{ \nabla A }^2 \leq -  2\delta \frac{\abs{ \nabla A }^2}{ (\abs{H}^2)^{1-\sigma} }.
\end{multline}
\end{proof}

As devised by Huisken \cite{Huisken1984}, we exploit the negative gradient term involving $\epsilon_{\nabla}$ with a Poincar\'e-type inequality, derived by integrating Simons' identity. Contracting the Simons identity \eqref{eqn:SimonsId} with $A_{ij}$ we obtain
\begin{align}
\frac 12 \Delta | A | ^ 2 & =  A _{ ij} \cdot \nabla_i\nabla_j H + | \nabla A |^2+ Z, \label{eqn:contractedSimons}
\end{align}
where 
\begin{align*}
Z = \sum_{ i,j,p,\alpha,\beta } H_\alpha h _{ ip\alpha } h _{ ij \beta } h _{ pj \beta }  - \sum_{ \alpha,\beta } \bigg( \sum_{i,j} h_{ ij \alpha } h _{ ji \beta} \bigg)^2 - | \Rp | ^ 2.
\end{align*}
A lower bound on $Z$ was obtained in \cite{Andrews2010, Baker2011} by an inelegant series of estimates that obscures the dependance of the nonlinearity on the submanifold intrinsic and normal curvature. Below, we provide a new estimate for surfaces immersed in $\mathbb{R}^4$ that makes transparent the dependence of the nonlinearity on the submanifold intrinsic and normal curvature. As alluded to in the introduction, the refined form of the Simons identity is of interest in the classification of minimal submanifolds of spheres. 

\begin{prop}\label{prop:newSimonsId}
For a two-dimensional submanifold $\Sigma^2$ immersed in $\mathbb{R}^4$, the nonlinearity in the contracted Simons identity satisfies
\begin{equation*}
	Z = 2 K | \Acirc|^2 - 2 |K^\perp|^2.
\end{equation*}
\end{prop}
\begin{proof}
The nonlinearity in the contracted Simons identity is
\begin{align*}
Z =  \sum_{ i,j,p,\alpha,\beta } H_\alpha h _{ ip\alpha } h _{ ij \beta } h _{ pj \beta }  -\sum_{ \alpha,\beta } \bigg( \sum_{i,j} h_{ ij \alpha } h _{ ij \beta} \bigg)^2 - |\Rp| ^ 2.
\end{align*}
Splitting the first term on the right into diagonal and off-diagonal summations, and using $h_{ij1} = 0$ for $i \neq j$, we get
\begin{align*}
\sum_{ i,j,p,\alpha,\beta } H_\alpha h _{ ip\alpha } h _{ ij \beta } h _{ pj \beta } &=   \sum_i h_{ii\alpha} \sum_{i,j} h_{ii\alpha} (h_{ii1})^2  + \sum_i h_{ii\alpha} \sum_{i,j} h_{ii\alpha} (h_{ii2})^2  \\
&\quad +  \sum_i h_{ii\alpha} \sum_{i \neq j} h_{ii\alpha} (h_{ij2})^2   + \sum_i h_{ii\alpha} \sum_{i \neq p} h_{ip\alpha} h_{ij\beta}h_{pj\beta}.
\end{align*}
The final term on the right is zero, as computing in the special orthonormal frames we see
\begin{align*}
\sum_i h_{ii\alpha} \sum_{i \neq p} h_{ip\alpha} h_{ij\beta}h_{pj\beta}&= H \, \sum_{i \neq p} h_{ip1} h_{ij\beta}h_{pj\beta} \\
& =0,
\end{align*}
since $h_{ip1} = 0$ for $i \neq p$.
We similarly split the second term on the right of Z into diagonal and off-diagonal sums, and putting all terms together we have
\begin{align*}
Z &= \sum_i h_{ii\alpha}  \sum_{i,j} h_{ii\alpha} (h_{ii1})^2 + \sum_i h_{ii\alpha}  \sum_{i,j} h_{ii\alpha} (h_{ii2})^2  + \sum_i h_{ii\alpha} \sum_{i \neq j} h_{ii\alpha} (h_{ij2})^2 \\ 
&\quad - \sum_{\alpha} \bigg( \sum_{i} h_{ii1} h_{ii\alpha} \bigg)^2 - \sum_{\alpha} \bigg( \sum_{i} h_{ii2} h_{ii\alpha} \bigg)^2 - \sum_{\alpha} \bigg( \sum_{i\neq j} h_{ij2} h_{ij\alpha} \bigg)^2\\ &\quad-2\sum_{\alpha,\beta}\bigg( \sum_{i=j} h _{ij\alpha}h_{ij\beta} \sum_{i\neq j}h_{ij\alpha}h_{ij\beta} \bigg)
 - |\Rp| ^ 2.\\
\end{align*}
We estimate these terms in pairs, gathering the first, second and third terms of lines one and two, respectively. Dealing with the first pair of terms, we follow \cite{Smyth1973} but keep track of the normal curvature terms, computing
\begin{align*}
\sum_i h_{ii\alpha} \sum_{i,j} h_{ii\alpha} (h_{ii1})^2  - \sum_{\alpha} \bigg( \sum_{i} h_{ii1} h_{ii\alpha} \bigg)^2 &= \left( K + \sum_{\alpha}(h_{12\alpha})^2 \right)(h_{111} - h_{221})^2 \\
&= K(4a^2) + 4a^2c^2.
\end{align*}
We estimate the second pair of terms in the same way, obtaining
\begin{align*}
\sum_i h_{ii\alpha}  \sum_{i,j} h_{ii\alpha} (h_{ii2})^2 -  \sum_{\alpha} \bigg( \sum_{i} h_{ii2} h_{ii\alpha} \bigg)^2 &= \left( K + \sum_{\alpha}(h_{12\alpha})^2 \right)(h_{112} - h_{222})^2 \\
&= K(4b^2) + 4b^2c^2.
\end{align*}
For the third pair of terms, as there are no diagonal terms to easily factor into the intrinsic curvature, we proceed by computing in the special orthonormal frames from the outset:
\begin{align*}
&\sum_i h_{ii\alpha}  \sum_{i \neq j} h_{ii\alpha} (h_{ij2})^2  - \sum_{\alpha} \bigg( \sum_{i\neq j} h_{ij2} h_{ij\alpha} \bigg)^2 \\
&\quad = 4c^2 \left(\frac{|H|^2}{4} - c^2\right) \\
&\quad= 4c^2 \left( \frac{|H|^2}{4} - a^2 - b^2 - c^2 \right) + 4c^2(a^2 + b^2) \\
&\quad= 4c^2 K + 4c^2(a^2 + b^2).
\end{align*}
Finally, as $h_{ij1} = 0$, the only non-zero contribution comes from $\alpha,\beta =2$ and we get
\begin{align*}
2\sum_{\alpha,\beta}\bigg( \sum_{i=j} h _{ij\alpha}h_{ij\beta} \sum_{i\neq j}h_{ij\alpha}h_{ij\beta} \bigg) &= 2\bigg( \sum_{i=j} h _{ij2}h_{ij2} \sum_{i\neq j}h_{ij2}h_{ij2} \bigg)\\
&=2(2 b^2)(2c^2) = 8 b^2c^2.
\end{align*}
Collecting all the terms together, and recalling $|\Rp|^2 = 16a^2c^2 = 4|K^\perp|^2$, we achieve
\begin{align*}
Z &= 2 K(2a^2 + 2b^2 + 2c^2) + 8a^2c^2 + 8b^2c^2 - 16a^2c^2-8 b^2c^2\\
&= 2K|\Acirc|^2 - 2|K^\perp|^2.
\end{align*}
\end{proof}
\begin{prop}[cf Lemma 5 \cite{Andrews2010}] 
For a two-dimensional submanifold $\Sigma^2$ immersed in $\mathbb{R}^4$, if the second fundamental form of $ \Sigma^2$ satisfies $ |A| ^ 2 < \nicefrac56 |H| ^ 2 $, then there exists a strictly positive constant $\epsilon_Z $ depending only on $ \Sigma _0 $ such that $ Z \geq \epsilon_Z ( |\Acirc|^2 + 2\gamma |K^{\perp}| ) |H| ^ 2 $.
\end{prop}
\begin{proof}
We can simply estimate $|K^{\perp}| \leq \nicefrac12 |\Acirc|^2$, in which case
\[ Z \geq 2 |\Acirc|^2 ( K - \frac14 |\Acirc|^2 ). \]
For a surface $K = \nicefrac12(|H|^2 - |A|^2)$, and therefore $K - \frac14 |\Acirc|^2 >0$ so long as $ |A| ^ 2 < \nicefrac56 |H| ^ 2$. The estimate can obviously be optimised by more careful use of the pinching inequality.
\end{proof}
The lower bound on $Z$ furnishes the following Poincar\'e-type inequality. The proof of this estimate is similar to the proof of the corresponding estimate in \cite{Huisken1984, Andrews2010}, except for the appearance of the Laplacian of the normal curvature. We only show how to deal with this last term and refer the reader to \cite{Andrews2010} and \cite{Baker2011} for the remainder of the calculations.
\begin{proposition}[cf Proposition 11 \cite{Andrews2010}]\label{prop_eleven} 
For every $ p \geq 2 $ and $ \eta > 0$ we have the estimate
\begin{equation}\label{eqn:PoincareType}
\int _{ \Sigma} f ^ p _ \sigma | H| ^ 2 d \mu _g \leq \frac { (4p\eta + 10 )}{ \epsilon_Z } \int _{ \Sigma} \frac { f ^ { p-1 }_ \sigma}{ | H| ^ { 2 ( 1- \sigma )}} |\nabla A | ^ 2 d \mu _ g  + \frac{3(p-1)}{\epsilon_Z\eta}\int_{\Sigma} f_{\sigma}^{p-2}\abs{\nabla f_{\sigma}}^2 \, d\mu_g.
\end{equation}
\end{proposition}
\begin{proof}
Using the contracted form of Simons' identity, the Laplacian of $f_{\sigma}$ can be expressed as
	\begin{align*}
		\Delta f_{\sigma} &= \frac{2}{\abs{H}^{2(1-\sigma)}}\big\langle \Acirc_{ij}, \nabla_i\nabla_j H \big\rangle + \frac{ 2 }{ \abs{H}^{2(1-\sigma)}} |\nabla \Acirc |^2 +  \frac{2}{\abs{H}^{2(1-\sigma)}}Z  \\
		&\quad  - \frac{2(1-\sigma)}{|H|^2} \langle \nabla_i |H|^2, \nabla_i f_{\sigma} \rangle - \frac{\sigma(1-\sigma)}{(|H|^2)^2} f_{\sigma}|\nabla |H||^2 - (1-\sigma)f_{\sigma}\Delta |H|^2 \\
		&\quad + \frac{ 2\gamma \Delta|K^{\perp}| }{ |H|^{2(1-\sigma) } }.
	\end{align*}
We now multiply by $f_{\sigma}^{p-1}$ and estimate the terms on the first two lines in the same manner as \cite{Andrews2010}, the only difference being we estimate in terms of $|\nabla A|^2$ instead of $| \nabla H |^2$, which is easily done as a final step by $|\nabla H|^2 \leq \nicefrac43 |\nabla A|^2$.
We now show how to deal with the new term on the last line involving the normal curvature. In the first step, we integrate and use Green's first identity to get
\begin{align}
	&\int \frac{ f^{p-1} \Delta |K^{\perp}| }{ |H|^{2(1-\sigma) } } \, d\mu_g  \notag \\ 
	&\quad= \int \nabla_i \left( \frac{ f^{p-1} }{ |H|^{2(1-\sigma)} } \right) \nabla_i |K^{\perp}| \, d\mu_g \notag   \\
	&\quad= (p-1)\int \frac{ f_{\sigma}^{p-2} \nabla_i f_{\sigma} \nabla_i |K^{\perp}| }{ |H|^{2(1-\sigma)} } \, d\mu_g - 2(1-\sigma)\int \frac{ f_{\sigma}^{p-1} \nabla_i |H| \nabla_i |K^{\perp}| }{ |H|^{2(1-\sigma)+1} } \, d\mu_g. \label{eqn:Poincare1}
\end{align}
Inspection of the formula for the normal curvature \eqref{eqn:normalCurv} reveals we can estimate $|\nabla K^{\perp}| \leq 4|\Acirc| |\nabla A|$. We use this last inequality and the Peter-Paul inequality to estimate equation \eqref{eqn:Poincare1} by
\begin{multline*}
	\int \frac{ f^{p-1} \Delta |K^{\perp}| }{ |H|^{2(1-\sigma) } } \, d\mu_g \\ \leq \frac{ 4(p-1) }{ \eta } \int  f_{\sigma}^{p-2} | \nabla f_{\sigma} |^2 \, d\mu_g + \left( 4(p-1)\eta + 10 \right)  \int  \frac{ f_{\sigma}^{p-1} | \nabla A |^2 }{ |H|^{2(1-\sigma) } } \, d\mu_g.
\end{multline*}
The proposition follows by combining this last estimate with the aforementioned estimates of \cite{Andrews2010}.
\end{proof}

The Poincar\'e-type inequality \eqref{eqn:PoincareType} allows us to prove sufficiently high $L^p$-norms of $f_{\sigma}$ are non-increasing in time, and crucially, that $\sigma$ decays like $1/\sqrt{p}$ as $p \rightarrow \infty$.
\begin{prop}
There exists constants $ c_3 $ and $ c_4$ depending on $ \Sigma _0 $ such that if $ p \geq c _ 3 $ and $ \sigma \leq \frac { c _4}{ \sqrt { p }} $ then for all time $ t \in [0,T)$ we have
\begin{equation*}
\frac{d}{dt} \int_{\Sigma} f_{\sigma}^p \, d\mu_g \leq 0.
\end{equation*}
\end{prop}
The following estimate, which states that higher powers of mean curvature can be absorbed into $f_{\sigma}$, depends on the refined decay enabled by the Poincar\'e-type inequality.

\begin{prop}\label{prop:powersOfHabsorbed}
There exists constants $c_5$ and $c_6$ depending only on $\Sigma_0$ such if $p \geq c_5$ and $\sigma \leq c_6/\sqrt{p})$, then for all time $ t \in [0,T)$ we have the estimate
	\begin{equation*}
		\int\limits_{\Sigma} |H|^n f_{\sigma}^p \, d\mu_g \leq \int\limits_{\Sigma}f_{\sigma'}^p \, d\mu_g.
	\end{equation*}
\end{prop}
With the last estimate in place, we can proceed by a Stampacchia iteration argument to bound $f_{\sigma}$ in $L^{\infty}$. We refer the reader to \cite{Huisken1984} for the details.

\section{Convergence to a round point}
The estimate of the previous section enables us to characterise the asymptotic shape of the evolving submanifolds as $t \rightarrow T$. We achieve this by performing a type I blowup and utilising the compactness theorem for mean curvature flow as proven in \cite{Baker2011}. The interested reader may like to compare the following argument with the corresponding argument for the Ricci flow, which can found, for example, in \cite{Topping2006}. Here the Codazzi equation performs the same role as the contracted second Bianchi identity, and the Codazzi Theorem that of Schur's Theorem.  For a proof of the Codazzi Theorem we refer the reader to \cite{Spivak1979}.

\begin{theorem}
Let $F : \Sigma^2 \times [0, T) \rightarrow \mathbb{R}^{4}$ be a solution of the mean curvature flow.  Assume that the initial submanifold $\Sigma_0$ is closed and satisfies $\abs{ H }_{min} > 0$ and $\abs{A}^2 + 2\gamma|K^{\bot}| \leq k\abs{H}^2$, where $ \gamma = 1 - \nicefrac43 k$ and $k \leq \nicefrac {29}{40}$. Then there exists a sequence of rescaled mean curvature flows $F_j : \Sigma^2 \times I_j \rightarrow \mathbb{R}^{4}$ containing a subsequence of mean curvature flows (also indexed by $j$) that converges to a limit mean curvature flow $F_{\infty} : \Sigma^2_{\infty} \times (-\infty, 0] \rightarrow \mathbb{R}^{4}$ on compact sets of $\mathbb{R}^{4} \times \mathbb{R}$ as $j \rightarrow \infty$. Moreover, the limit mean curvature flow is a shrinking sphere.
\end{theorem}
\begin{proof}
Pick any sequence of times $(t_j)_{j \in \mathbb{N}}$ such that $t_j \rightarrow T$ as $j \rightarrow \infty$. Proposition \ref{prop:pinching} implies that $\abs{ A }^2$ and $\abs{ H }^2$ have equivalent blow-up rates, so we can in fact rescale by $\abs{ H }^2$.  Since $\Sigma^2$ is assumed to be closed, we can pick a sequence of points $(p_j)_{ j \in \mathbb{N} }$ defined by
\begin{equation*}
	\abs{H}(p_j, t_j)  = \max_{ p \in \Sigma^2 }\abs{ H }( p, t_j ).
\end{equation*}
For convenience, set $\lambda_j := \abs{H}(p_j, t_j)$ and define a sequence of rescaled and translated flows by
\begin{equation*}
	F_j(q,s) = \lambda_j \big( F(q, t_j + s/\lambda_j^2) - F(p_j,t_j) \big).
\end{equation*}
It is easily checked this is a parabolic rescaling, and consequently for each $j$, the rescaled flow $F_j: \Sigma^2 \times [\lambda_j^2T, 0] \rightarrow \mathbb{R}^{4}$ is a solution of the mean curvature flow (in the time variable $s$). The second fundamental form of the rescaled flows is uniformly bounded above independent of $j$ and we can apply the compactness theorem for mean curvature flows (see \cite{Baker2011}) to obtain a smooth limit solution of the mean curvature flow $F_{\infty} : M_{\infty} \times (-\infty, 0] \rightarrow \mathbb{R}^{4}$. Furthermore, by construction of the sequence $F_j$, the limit solution satisfies $\abs{ H }^2_{\infty}(\cdot,0) = 1$ at some point. The estimate of Theorem \ref{eqn:tracelessEst} rescales as
\begin{equation*}
	|\Acirc|^2_j + 2\gamma |K^{\perp}|_j \leq c_0 \lambda_j^{ -\delta }\abs{ H }^2_j,
\end{equation*}
and upon sending $j \rightarrow \infty$ we find
\begin{equation*}\label{e: blow up 1}
	\abs{ \Acirc }^2_{ \infty } + 2\gamma |K^{\perp}|_{\infty}  = 0.
\end{equation*}
The previous line implies that $|\Acirc|_{\infty}^2 = 0$ and hence $F_{\infty}(M_{\infty},t)$ is totally umbilic.  By the Codazzi Theorem, $F_{\infty}(M_{\infty},t)$ is a plane or a $2$-sphere lying in a $3$-dimensional affine subspace of $\mathbb{R}^{4}$. We know the limit solution has positive mean curvature at some point and therefore must be a sphere.
\end{proof}

\end{document}